\documentclass[reqno]{amsart}

\usepackage[english]{babel} 
\usepackage{amscd}
\usepackage{amssymb}
\usepackage{amsfonts}
\usepackage{amsmath} 
\usepackage{amsthm}
\usepackage{amsxtra}
\usepackage{pb-diagram} 
\usepackage{latexsym} 

\newcommand{\R}{\mathbb{R}} 
\newcommand{\N}{\mathbb{N}}
\newcommand{\Z}{\mathbb{Z}}

\newcommand{\sign}{\mathop\mathrm{sign}\nolimits}

\newcommand{\cl}[1]{\overline{#1}}

\newcommand{\D}{\partial}
\newcommand{\X}{\times} 

\newcommand{\h}[1]{\widehat{#1}} 
\newcommand{\su}[1]{\overline{#1}}

\newcommand{\rank}{\mathop\mathrm{rank}\nolimits}
\newcommand{\im}{\mathop\mathrm{im}\nolimits}

\newcommand{\rmx}{\mathrm{x}}
\newcommand{\xbf}{\boldsymbol{\mathrm{x}}}

\newcommand{\rkrs}{\mathbb{R}^k\times\mathbb{R}^s}

\newcommand{\ctrkrs}{C_T (\R^k\times\R^s)}

\newtheorem{theorem}{Theorem}[section]
\newtheorem{corollary}[theorem]{Corollary}
\newtheorem{lemma}[theorem]{Lemma}
\newtheorem{example}[theorem]{Example}
\newtheorem{proposition}[theorem]{Proposition}
\newtheorem{remark}[theorem]{Remark}

\numberwithin{equation}{section}
 
\begin{document}

\title[On a class of nonautonomous differential-algebraic equations]
{Harmonic solutions to a class of differential-algebraic equations
  with separated variables}

\author{Luca Bisconti} \address[L.\ Bisconti]{Dipartimento di Sistemi
  e Informatica, Universit\`a di Firenze, Via Santa Marta 3, 50139
  Firenze, Italy} \email{luca.bisconti@unifi.it}

\begin{abstract} We study the set of $T$-periodic solutions of a class
  of $T$-periodically perturbed Differential-Algebraic Equations with
  separated variables. Under suitable hypotheses, these equations are
  equivalent to separated variables ODEs on a manifold.  By combining
  known results on Differential-Algebraic Equations, with an argument
  about ODEs on manifolds, we obtain a global continuation result for
  the $T$-periodic solutions to the considered equations.  As an
  application of our method, a multiplicity result is provided.
\end{abstract}


\maketitle
\section{Introduction} This paper is concerned with studying
properties of the set of harmonic solutions to periodic perturbations
of a class of $T$-periodic \emph{Differential-Algebraic Equations}
(DAEs). More precisely, we will consider the following separated
variables DAE in semi-explicit form
\begin{equation} \label{SVDAEs:0} \left\{ \begin{array}{l} \dot x =
      a(t)f(x, y), \\ g(x, y) = 0, \\
    \end{array} \right.
\end{equation} where $f\colon U\to\R^k$ and $g\colon U\to\R^s$ are
given continuous maps defined on an open connected set $U\subseteq
\rkrs$ and $a : \R \to \R$ is continuous.  Assume that $g\in
C^{\infty}(U, \R\sp{s})$, with the property that the
Jacobian matrix $\D_2 g (p,q)$ of $g$, with respect to the last $s$
variables, is invertible for any $(p,q)\in U$.  
Given $T > 0$, we also assume that the function $a$ is $T$-periodic and
such that
\begin{equation} \label{SVDAEs:average-nn} \frac{1}{T}\int^T_0 a(t)
  dt\ne 0,
\end{equation} and we consider the perturbed problem
\begin{equation} \label{SVDAEs:0a} \left\{ \begin{array}{l} \dot x =
      a(t)f(x, y) +\lambda h(t, x, y), \,\,\, \lambda\geq 0,\\
      g(x, y) = 0, \\
    \end{array} \right.
\end{equation} 
where $h\colon\R\X U\to\R^k$ is continuous and $T$-periodic in the
first variable.  Here, the function $a$ has the meaning of a
perturbation factor and, up to dividing $a$ by its average, we can
assume that $a = 1 + \alpha $, where $\alpha : \R \to \R$ is
continuous, $T$-periodic and with zero average. Sometimes, for
convenience, we will refer to the pair of maps $(\alpha, h)$ as the
$T$-\emph{perturbation pair}.

Since $\D_2 g(p,q)$ is invertible for any $(p, q)\in U$, then $0$ is a
regular value for $g$, and so $g^{-1}(0)$ is a closed $C^{\infty}$
submanifold of $U$ of dimension $k$. Throughout this article the
letter $M$ will be used to denote the submanifold $g^{-1}(0)$ and, if
not otherwise stated, the points of $M$ will be written as pairs
$(p,q)\in M$.

As a direct consequence of the Implicit Function Theorem, $M$ can be
locally represented as a graph of some map from an open subset of
$\R^k$ to $\R^s$ and, hence, Equation \eqref{SVDAEs:0a} can be locally
decoupled. However, globally, this might be false or not convenient
for our purpose (see, e.g. \cite{cala, spaDAE}).  For instance, even
in the case when $M$ is a graph of some map $\gamma$, the analytical
expression of $\gamma$ might be so complicated that the decoupled
version of \eqref{SVDAEs:0a} turns out impractical.

Following the ideas in \cite{CaSp, spaDAE} (see also \cite{RR}), we
will show that the Equation \eqref{SVDAEs:0a} is equivalent to an ODE
on $M$. Therefore, to study the $T$-periodic solutions to
\eqref{SVDAEs:0a} we will be able to use topological methods based on
the degree of tangent vector fields, as well as results on periodic
solutions to ODEs on manifolds.

Our aim is to get information on the structure of the set of solutions
to \eqref{SVDAEs:0a}.  Thus, denoted by $C_T(U)$ the metric subspace
of $C_T (\rkrs)$ of all the continuous $T$-periodic functions taking
values in $U$, we will give conditions ensuring the existence of a
connected set of pairs $(\lambda, (x, y))\in [0, \infty)\X C_T(U)$,
with $\lambda> 0$ and $(x, y)$ a $T$-periodic solution to
\eqref{SVDAEs:0a}, whose closure is not compact and meets the set of
pairs formed by constant solutions corresponding to $\lambda=0$ (the
so called \emph{trivial} pairs).  \medskip

\noindent\textbf{Plan of the paper.}  In Section \ref{SVDAEs:section2}
we introduce some preliminaries and we give an overview of our
approach to treat the above class of DAEs.  In section
\ref{SVDAEs:section3}, proceeding as in \cite{spaSepVar, spaDAE}, we
obtain information about the set of $T$-pairs to \eqref{SVDAEs:0a};
examples of applications of our method are given.  Thereafter, in
Section \ref{SVDAEs:section4}, a multiplicity result for the Equation
\eqref{SVDAEs:0a} is provided.

\section{Preliminaries and basic notions} \label{SVDAEs:section2} We
first recall some facts and definitions about the function spaces used
in the sequel.  Let $I \subseteq \R$ be an interval and $V \subseteq
\rkrs$ be open.  Given $r \in \N$, the set of all $V$-valued
$C^r$-functions defined on $I$ is denoted by $C^r(I,V)$.  Frequently,
we also use the notation $C(I,V)$ in place of $C^0(I,V)$ and, when
$I=\R$, we write $C^r(V)=C^r(\R, V)$.  We will denote by the symbol
$C_T (V)$ the metric subspace of the Banach space $C_T (\rkrs)$ of all
the continuous $T$-periodic functions assuming values in $V$.

Let us now give the precise concept of solution to equations of the
form \eqref{SVDAEs:0a}.  Let $U \subseteq \rkrs$ be open and
connected, let $g : U \to \R^s$, $f : U \to \R^k$, $a \colon \R\to \R$
and $h\colon\R\times U\to\R\sp{k}$ be given continuous maps, where
$g\in C^{\infty}(U, \R\sp{s})$ has the property that $\D_2 g (p,q)$ is
invertible for any $(p,q)\in U$.  For a given $\lambda\geq 0$, a
solution of \eqref{SVDAEs:0a} consists of a pair of functions $(x,
y)\in C^1(I,U)$, with $I\subseteq\R$ an open interval, such that
\begin{equation*} \left\{
    \begin{array}{l} \dot x(t) = a(t)f\big(x(t), y(t)\big) + \lambda
      h(t, x(t), y(t)), \\ g(x(t), y(t))=0,
    \end{array} \right.
\end{equation*} for all $t\in I$. 
Existence and uniqueness results for the initial value problems
related to DAEs of type \eqref{SVDAEs:0a} will be discussed in 
the next section:
they will be deduced as consequences of known theorems about 
ODEs on manifolds.

Let $(x,y)$ be a solution of \eqref{SVDAEs:0a} defined on an interval
$I$, and corresponding to a given $\lambda\geq 0$. We say that $(z,
w)$ is a continuation (or extension) of $(x, y)$ if there exists an
interval $\tilde{I}\supset I$, such that $(z, w)$ coincides with $(x,
y)$ on $I$, and $(z, w)$ satisfies \eqref{SVDAEs:0a}. A solution $(x,
y)$ for which there are no extensions is called \emph{maximal}.  As in
the case of ODEs, Zorn's lemma implies that any solution is the
restriction of a maximal solution.

\begin{remark} \label{first-remark} Let $U\subseteq \rkrs$ be open and
  connected.  Consider the following initial value problems
  \begin{equation} \label{SVDAEs:01aaa} \emph{\text{(A)}}\quad
    \left\{ \begin{array}{l} \dot x = a(t)f(x, y), \\ g(x, y)=0, \\
        (x(0),y(0))=(x_0,y_0),
      \end{array} \right.  \quad \text{and}\quad\,\,
    \emph{(A}^{\prime}\emph{)} \quad \left\{ \begin{array}{l} \dot x =
        f(x, y), \\ g(x, y)=0, \\ (x(0),y(0))=(x_0,y_0),
      \end{array} \right.\!\!
  \end{equation} 
  where $(x_0,y_0)\in U$, $f : U \to \R^k$, $g : U \to \R^s$ are
  defined above, and the map $a : \R \to \R$ is continuous,
  $T$-periodic with $1/T \int_0^T a(t) dt = 1$.  Suppose also that $f$
  is $C^1$, so that the uniqueness of solutions for the problems $
  \emph{(A)}$ and $\emph{(A}^{\prime}\emph{)}$ is guaranteed (see
  below for details). Let $(\xi, \sigma) : J \to U$ and $(u,v) : I \to
  U$ be the maximal solutions of \emph{(A)} and
  \emph{($\text{A}^\prime$)} respectively, with $I$ and $J$ the
  related maximal intervals of existence.

  Let $t>0$ be such that $\int_0^{l} a(s) ds \in I$ for all $l\in [0,
  t]$, then it follows that
  \begin{equation*} \label{SVDAEs:02} (\xi(t), \sigma(t)) = \bigg(
    u\Big(\int_0^t a(s) ds\Big), v\Big(\int_0^t a(s) ds\Big)\bigg),
  \end{equation*} 
  and hence $t \in J$.  Conversely, by using a standard maximality
  argument, one can prove that $t \in J$ implies $\int_0^ta(s)ds\in
  I$.  Define the map $\phi_a: J\to I$, $t\mapsto
  \phi_a(t)=\int^t_0a(s)ds$.  Notice that, if $T\in J$, then
  $\phi_a(T)= T \in I$, and so $(\xi(T), \sigma(T)) = (x(T), y(T))$.
\end{remark}

The argument in Remark \ref{first-remark} plays a fundamental role in
our approach to treat Equation \eqref{SVDAEs:0a}.  As we will formally
prove, the study of the $T$-periodic solutions to \eqref{SVDAEs:0a}
can be reduced to that of the $T$-periodic solutions for a particular
autonomous semi-explicit DAE.  As a consequence, the main results in
\cite{spaDAE} can be applied to the present case.  \smallskip

\subsection{Associated vector Fields and ODEs on $M$}
Proceeding as in \cite{cala, spaDAE}, we associate to
\eqref{SVDAEs:0a} an ODE on $g^{-1}(0)=M$.

Let us first recall that, given a differentiable manifold
$N\subseteq\R^n$, a continuous map $w\colon N\to\R^n$ with the
property that for any $p\in N$, $w(p)$ belongs to the tangent space
$T_pN$ of $N$ at $p$ is called a \emph{tangent vector field on $N$}.

With our hypotheses it is always possible to associate a pair $\Psi,
\Upsilon$ of tangent vector fields on $M$ to the functions $f$ and $h$
in \eqref{SVDAEs:0a}. In fact, consider the maps $\Psi\colon
M\to\rkrs$ and $\Upsilon \colon \R\X M \to \rkrs$ defined as follows:
\begin{subequations}\label{campiv}
  \begin{align} &\Psi (p, q)= \big( f(p, q), -[\D_2 g(p, q)]^{-1} \D_1
    g (p, q) f(p, q) \big),\,\, \textrm{and} \label{SVDAEs:03} \\
    &\Upsilon (t, p, q)= \big( h(t, p, q), -[\D_2 g( p, q )]^{-1} \D_1
    g (p, q) h(t, p, q) \big). \label{SVDAEs:04}
  \end{align}
\end{subequations} Given a point $(p,q)\in M\subseteq\rkrs$, since
$T_{(p,q)}M$ coincides with the kernel $\ker d_{(p,q)}g$ of the
differential $d_{(p,q)}g$ of $g$ at $(p, q)$, it can be easily seen
that $\Psi$ is tangent to $M$ in the sense that $\Psi(p,q)$ belongs to
$T_{(p,q)}M$ for all $(p,q)\in M$ (see, e.g.\
\cite{spaDAE}). Analogously, the time-dependent vector field
$\Upsilon$ is tangent to $M$, i.e. $\Upsilon(t, p, q)\in T_{(p, q)}M$,
for all $(t, p, q)\in\R\X M$.

We claim that \eqref{SVDAEs:0a} is equivalent to the following ODE on
$M$, which implicitly keeps track of the condition $g(x,y)=0$:
\begin{equation} \label{SVDAEs:eq-su-M} \dot \zeta = a(t)\Psi (\zeta)
  + \lambda \Upsilon (t, \zeta),
\end{equation} meaning that $\zeta=(x,y)$ is a solution of
\eqref{SVDAEs:eq-su-M}, in an interval $I\subseteq\R$, if and only if so
is $(x,y)$ for \eqref{SVDAEs:0a}.  
We prove the claim:
for a given $\lambda>0$, let $(x, y)$ be a solution to
\eqref{SVDAEs:0a} defined on $I\subseteq \R$. Differentiating the
algebraic condition $g\big(x(t),y(t)\big)=0$, one obtains
\begin{equation*} 0=\partial_1 g(x(t),y(t))\dot x (t)+ \partial_2
  g(x(t),y(t))\dot y (t),
\end{equation*} whence
\begin{equation*} \dot y(t) = -[\D_2 g(x(t), y(t))]^{-1} \D_1 g (x(t),
  y(t)) \left[ a(t)f(x(t), y(t)) + \lambda h(t, x(t), y(t))\right],
\end{equation*} with $t\in I$. Then, the solutions of
\eqref{SVDAEs:0a} correspond to those of \eqref{SVDAEs:eq-su-M}.
Conversely, if $\zeta=(x,y)$ is a solution of \eqref{SVDAEs:eq-su-M}
defined on an interval $I\subseteq \R$, then it satisfies identically
$\big(x(t),y(t)\big)\in M$, which implies $g(x(t),y(t))=0$, and
fulfills
\begin{equation*} \dot\zeta(t)=a(t)\Psi\big(\zeta(t)\big)+\lambda
  \Upsilon (t,\zeta(t)),\,\, \forall t\in \R,
\end{equation*} whose first component is \eqref{SVDAEs:0a}. 
\smallskip

Now, consider the unperturbed version of Equation
\eqref{SVDAEs:eq-su-M}:
\begin{equation} \label{SVDAEs:eq-su-M-unpert} \dot \zeta = a(t)\Psi
  (\zeta).
\end{equation}
By the definition of the vector field $\Psi$, if $f$ is $C^1$, then
$\Psi$ is $C^1$ too. This condition ensures that any initial value
problem associated to \eqref{SVDAEs:eq-su-M-unpert} admits a unique
solution.  Then, as a consequence of the equivalence of
\eqref{SVDAEs:0} with \eqref{SVDAEs:eq-su-M-unpert}, the local results
on existence, uniqueness and continuous dependence of local solutions
of the initial value problems translate to \eqref{SVDAEs:0} from the
theory of ODEs on manifolds (see, e.g. \cite{Lang-1}).  Observe that,
if also $h$ is $C^1$, a similar statement holds for \eqref{SVDAEs:0a}
and \eqref{SVDAEs:eq-su-M}. \medskip

Let $N\subseteq \R^n$ be a differentiable manifold and let $\Xi \colon
\R\X N \to \R^n$ be a time-dependent tangent vector field sufficiently
regular in order to guarantee the existence and uniqueness of the
solutions for the initial value problems associated to the
differential equation
\begin{equation}
  \label{SVDAEs:eq-test}
  \dot \zeta =\Xi (t, \zeta),\,\,\, t\in\R.
\end{equation} 
Denote by
\begin{equation*}
  \begin{aligned}
    \mathcal{D}=\big\{ (\tau, p) \in \R \X N \colon \,\, \textrm{the
      solution of}\,\, & \eqref{SVDAEs:eq-test} \,\, \textrm{which
      satisfies}\,\,
    \zeta(0)=p\\
    &\textrm{is continuable at least up to}\,\, t=\tau \big\}.
  \end{aligned}
\end{equation*}
A well known argument based on some global continuation properties of
the flows (see, e.g. \cite{Lang-1, Lang-2}) shows that $\mathcal{D}$
is an open set containing $\{0\}\X M$.  Let $P^\Xi \colon \mathcal{D} \to
N$ be the map that associates to each $(t, p)\in \mathcal{D}$ the
value $\zeta (t)$ of the maximal solution $\zeta$ to
\eqref{SVDAEs:eq-test} such that $\zeta (0) = p$ (i.e.  $ P^\Xi (t, p)
= \zeta (t) $).  Here and in the sequel, we will denote by
$P^\Xi_\tau, \tau\in \R$, the local (Poincar\'e) $\tau$-translation
operator associated to the equation \eqref{SVDAEs:eq-test}.  It holds
true that $P^\Xi_\tau(p)=P^\Xi(\tau, p)$, with $(\tau,p)\in
\mathcal{D}$.  Therefore, the domain of $P^\Xi_\tau$ is an open set
formed by the points $p \in N$ for which the maximal solution of
\eqref{SVDAEs:eq-test}, starting from $p$ at $t = 0$, is defined up to
$\tau$.

\begin{remark} \label{SVDAEs:remark-on-ivp} By virtue of the
  equivalence of \eqref{SVDAEs:eq-su-M-unpert} with \eqref{SVDAEs:0},
  the initial value problems $\emph{(A)}$ and $\emph{(A}'\emph{)}$ in
  \eqref{SVDAEs:01aaa} respectively become:
  \begin{equation*} \label{SVDAEs:equivalence-on-manifold}
    \emph{\text{(B)}}\quad \left\{ \begin{array}{l} \dot \zeta =
        a(t)\Psi(\zeta), \\ \zeta(0)=\zeta_0,
      \end{array} \right.  \quad \text{and}\quad\,\,
    \emph{\text{(B}}'\emph{\text{)}} \quad \left\{ \begin{array}{l} \dot
        \zeta = \Psi(\zeta), \\ \zeta(0)=\zeta_0,
      \end{array} \right.
  \end{equation*} with $\zeta_0=(p_0, q_0)\in M$.  Let $J$ and $I$ be the
  intervals on which are defined the (unique) maximal solutions of
  $\emph{(B)}$ and $\emph{(\text{B}}'\emph{)}$ respectively.  As
  before, consider the map $\phi_a: J\to I$, $t\mapsto
  \phi_a(t)=\int^t_0a(s)ds$ and assume $T\in J$, so that $\phi_a(T)=T\in
  I$.  Then, also in this case, the map $\phi_a$ allows us to
  write the solution of $\emph{(B)}$ in terms
  of the solution of $\emph{(B}'\emph{)}$ on the interval $[0,T]$. 
  Rephrasing all of this in terms of the local Poincar\'e operators 
  $P^{a\Psi}_t, P^{\Psi}_t$, $t\in \R$,
  associated to $\emph{(B)}$ and $\emph{(\text{B}}'\emph{)}$, it holds true
  that: if 
  $P^{a\Psi}_T(\zeta_0)$ is defined,  then $P^\Psi_T (\zeta_0)$ is defined 
  too and, in such a case, $P^\Psi_T (\zeta_0) = P^{a\Psi}_T (\zeta_0)$.  
\end{remark}

To conclude this subsection, we give some conventions that will be
widely used in the sequel.  Let $N\subseteq \R^n$ be a differentiable
manifold, and $T > 0$ positive number.  We denote by $C_T (N)$ the
metric subspace of the Banach space $C_T (\R^n)$ of all the
$T$-periodic continuous functions $\xi : \R \to N$.  Notice that $C_T
(N)$ is not complete unless $N$ is closed in $\R^n$.  Consider the
following diagram of closed embeddings:
\begin{equation} \label{SVDAEs:graph}
  \begin{diagram} \node{[0,\infty)\X N} \arrow{e} \node{[0,\infty)\X
      C_T(N)} \\ \node{N} \arrow{e} \arrow{n} \node{C_T(N)} \arrow{n}
  \end{diagram}
\end{equation} which allow us to identify any space in the above
diagram with its
image.  In particular, $N$ will be regarded as its image in $C_T (N)$
under the embedding that associates to any $p\in N$ the map
$\overline{p} \in C_T (N)$ constantly equal to
$p$. Furthermore, we will regard $N$ as the slice $\{0\}\X N
\subseteq [0, \infty)\X N$ and, analogously, $C_T (N)$ as $\{0\}\X C_T
(N)$.  Thus, if $\Omega$ is a subset of $[0, \infty)\X C_T(N)$, then
$\Omega\cap N$ represents the set of points of $N$ that, regarded as
constant functions, belong to $\Omega$. Namely, we have that
\begin{equation} \label{SVDAEs:convention-constan-func} \Omega\cap
  N=\big\{\zeta \in N : (0,\overline{p})\in\Omega\big\}.
\end{equation}

\subsection{The degree of the tangent vector field $\Psi$ and some
  related properties} We now give some notions about the degree of
tangent vector fields on manifolds. Recall that if $w:N\to\R^n$ is a
tangent vector field on the differentiable manifold $N\subseteq\R^n$
which is (Fr\'echet) differentiable at $p\in N$ and $w(p) = 0$, then
the differential $\textrm{d}_{p} w \colon T_{p}N \rightarrow \R^n$
maps $T_{p}N$ into itself (see, e.g.\ \cite{milnor}), so that, the
determinant $\det\, \textrm{d}_{p}w$ is defined.  In the case when
${p}$ is a nondegenerate zero (i.e.\ $\textrm{d}_{p} w \colon T_{p} N
\rightarrow \R^n$ is injective), ${p}$ is an isolated zero and $\det\,
\textrm{d}_{p} w \ne 0$. Let $W$ be an open subset of $N$ in which we
assume $w$ admissible for the degree, that is we suppose the set
$w^{-1}(0)\cap W$ is compact. Then, it is possible to associate to the
pair $(w, W)$ an integer, $\deg (w, W)$, called the degree (or
characteristic) of the vector field $w$ in $W$ (see e.g.\
\cite{FU-PE:1986,difftop}), which, roughly speaking, counts
(algebraically) the zeros of $w$ in $W$ in the sense that when the
zeros of $w$ are all nondegenerate, then the set $w^{-1}(0)\cap W$ is
finite and
\begin{equation} \label{SVDAEs:deg} \deg(w, W) = \sum_{q \in
    w^{-1}(0)\cap W} \sign\, \det\, \textrm{d}_{{q}} w.
\end{equation}
The concept of degree of a tangent vector field is related to the
classical one of Brouwer degree (whence its name), but the former
notion differs from the latter when dealing with manifolds.  In
particular, this notion of degree does not need the orientation of the
underlying manifolds. However, when $N=\R^n$, the degree of a vector
field $\deg(w, W)$ is essentially the well known Brouwer degree of $w$
on $W$ with respect to $0$.  The degree of a tangent vector field
enjoys, for instance, the following properties: \emph{Additivity,
  Excision, Homotopy invariance, Invariance under diffeomorphisms} and
\emph{Solution}.  For an exhaustive exposition of this topic, we refer
e.g.\ to \cite{FU-PE:1986, difftop, milnor}.

The Excision property allows the introduction of the notion of
\emph{index} of an isolated zero of a tangent vector field.  Let $q
\in N$ be an isolated zero of a tangent vector field $w \colon N \to
\R^n$. Clearly, $\deg(w, V )$ is well defined for any open set $V
\subseteq N$ such that $V \cap w^{-1}(0) = \{q\}$.  Moreover, by the
Excision property, the value of $\deg(w, V)$ is constant with respect
to such $V$'s. This common value of $\deg(w, V )$ is, by definition,
the index of $w$ at $q$, and is denoted by $\textrm{i} (w, q)$.  Using
this notation, if $(w,W)$ is admissible, by the Additivity property we
have that if all the zeros in $W$ of $w$ are isolated, then
\begin{equation*}
  \deg(w, W) = \sum_{q\in w^{-1}(0)\cap W}
  \textrm{i} (w, q).
\end{equation*}
By formula \eqref{SVDAEs:deg} we have that if $q$ is a nondegenerate
zero of $w$, then $\textrm{i} (w, q) = \sign \det d_qw$. \smallskip

Take $U \subseteq \rkrs\cong\R^n$ open and connected set.  Let
$g\colon U\to\R\sp{s}$ and $f\colon\R\times U\to\R\sp{k}$ be given
maps such that $f$ is continuous and $g$ is $C^{\infty}$ with the
property that $\partial_2 g (p, q)$ is invertible for all $(p,q)\in
U$. Let $\Psi$ be the tangent vector field on $M=g^{-1}(0)$ given by
\eqref{SVDAEs:03}. As we will see, a key requirement for the rest of
the paper is that the degree of $\Psi$ is nonzero. Define the map $F
\colon U \to \R^k \X \R^s$, as follows
\begin{equation}\label{SVDAEs:F}
  F(p, q) := \big(f(p, q), g(p, q)\big).
\end{equation} 
The following result \cite[Theorem\ 4.1]{spaDAE} allows us to reduce the
computation of the degree of the tangent vector field $\Psi$ on $M$ to
that of the Brouwer degree of the map $F$ with respect to $0$, which
is in principle handier. Namely, we have that

\begin{theorem} \label{SVDAEs:teo1} Let $U \subseteq \rkrs$ be open
  and connected, and let $F\colon U\to\rkrs$ be given by
  \eqref{SVDAEs:F}.  Then, for any $V\subseteq U$ open, if either
  $\deg(\Psi, M\cap V)$ or $\deg(F, V)$ is well defined, so is the
  other, and
  \begin{equation*} |\deg (\Psi, M\cap V)| = |\deg (F, V)|.
  \end{equation*}
\end{theorem}

\section{Connected sets of $T$-periodic solutions}
\label{SVDAEs:section3} The focus of this section is the study of the
$T$-periodic solutions to \eqref{SVDAEs:0a}, where $T>0$ is given. In
particular, we dwell on the topological structure of the set of
$T$-periodic solutions to this DAE.

Recall that $U \subseteq \R^k \X \R^s$ is open and connected, $f
\colon U \to \R^k, g \colon U \to \R^s$, $a \colon \R \to \R$ and $h
\colon \R\X U \to \R^k$ are continuous. We also assume that $a$ is
$T$-periodic with nonzero average, that $h$ is locally Lipschitz and
$T$-periodic in the first variable, $f\in C^1$, and $g$ is
$C^{\infty}$ with the property that $\det \D_2g(p, q) \ne 0$ for all
$(p, q) \in U$.

In the following, we say that $(\lambda, (x,
y))\in[0,\infty)\X\ctrkrs$ is a $T$-\emph{periodic pair} to
\eqref{SVDAEs:0a}, if $(x, y)$ is a $T$-periodic solution of
\eqref{SVDAEs:0a} corresponding to $\lambda$.  According to the
convention introduced in
\eqref{SVDAEs:graph}-\eqref{SVDAEs:convention-constan-func}, any $(p,
q) \in U$ will be identified with the element $(\overline{p},
\overline{q})$ of $C_T (U)$ that is constantly equal to $(p, q)$.  Let
$F \colon U \to \R^k \X \R^s$ be given by \eqref{SVDAEs:F}.  Since
$a(t)$ is not identically zero, a point $(p,q)\in U$ corresponds to a
constant solution $(\overline{p}, \overline{q})$ of \eqref{SVDAEs:0a},
for $\lambda = 0$, if and only if $F(p, q) = (0, 0)$.  A $T$-periodic
pair of this form will be called \emph{trivial}.  Thus, with this
notation, the set of trivial $T$-periodic pairs can be written as
\begin{equation*} \big\{(0, (\overline{p}, \overline{q})) \in
  [0,\infty) \X C_T (U) \colon F(p, q) = (0, 0)\big\}.
\end{equation*} 
Observe that, in the case $\lambda=0$, we may have the existence of
nontrivial $T$-periodic pairs to the system of equations
\eqref{SVDAEs:0a}.

Given $ \Omega \subseteq [0,\infty) \X C_T (U)$, with $\Omega \cap U$
we denote the subset of $U$ whose points, regarded as constant
functions, lie in $\Omega$. Namely, one has that $\Omega \cap U =
\big\{(p, q) \in U \colon (0, (\overline{p}, \overline{q})) \in\Omega
\big\}$.\medskip

We will make use of the following result which is a direct consequence
of \cite[Theorem\ 3.3]{spaSepVar}

\begin{theorem}\label{UT:Theorem-3.3-spaSepVar}
  Let $a \colon \R \to \R$ be a continuous function and let   
  $\omega \colon N \to \R^k$  and $\varrho \colon \R\X N \to \R^k$ be two
  continuous tangent vector fields on the boundaryless manifold
  $N\subseteq \R^k$. Consider the following parametrized differential 
  equation on $N$
  \begin{equation} \label{SVDAEs:eqeqeq} \dot \zeta = a(t)\omega
    (\zeta) + \lambda \varrho (t, \zeta),\,\, \lambda\geq 0.
  \end{equation}
  Assume that $\varrho$ and $a$ are $T$-periodic, with
  $1/T\int_0^Ta(t)dt \ne 0$. Let $\Sigma$ be an open subset of $[0,
  \infty) \X C_T (N)$, and assume that the degree $\deg(\omega,
  \Sigma\cap N)$ is well-defined and nonzero. Then  $\Sigma$ contains a
  connected set of nontrivial $T$-periodic pairs of
  \eqref{SVDAEs:eqeqeq} whose closure in $\Sigma$ is noncompact and
  meets $\omega^{-1}(0)\cap \Sigma$.
\end{theorem}

\begin{remark} Until now, condition \eqref{SVDAEs:average-nn} has been
  used just to characterize some properties of the map $\phi_a$
  introduced in Remark \ref{first-remark}.  However, even if it will
  not be shown explicitly here, this assumption has some other
  important implications in our method. Indeed, it is crucial for the
  proof of Theorem \ref{UT:Theorem-3.3-spaSepVar}, which in turn is
  the basis of the main result of this section.
\end{remark}

We state and prove the following Theorem

\begin{theorem} \label{SVDAEs:main-thm} Let $U \subseteq \R^k \X \R^s$
  be open and connected.  Let $g : U \to \R^s$, $f : U \to \R^k$, $a :
  \R\to \R$ and $h : \R \X U \to \R^k$ be as above.  Let also $F(p, q)
  = \big(f(p, q), g(p, q)\big)$ be defined as in \eqref{SVDAEs:F}.
  Given $\Omega \subseteq [0,\infty) \X C_T (U)$ open, assume that
  $\deg(F,\Omega\cap U)$ is well-defined and nonzero.  Then, there
  exists a connected set of nontrivial $T$-periodic pairs of
  \eqref{SVDAEs:0a} whose closure in $\Omega$ is noncompact and meets
  the set $\big\{(0, (\overline{p}, \overline{q})) \in \Omega : F(p,
  q) = (0, 0)\big\}$ of the trivial $T$-periodic pairs of
  \eqref{SVDAEs:0a}.
\end{theorem}
\begin{proof}
  In Section \ref{SVDAEs:section2} it has been shown that the system
  of equations \eqref{SVDAEs:0a} is equivalent to an ODE of type
  \eqref{SVDAEs:eq-su-M} on $M=g^{-1}(0)$.  Let $\Psi$ and $\Upsilon$
  be the tangent vector fields defined in \eqref{campiv}.  Let also
  $\mathcal{O}$ be the open subset of $[0,\infty)\X C_T(M)$ given by
  \begin{equation*} \mathcal{O}=\Omega\cap \Big([0,\infty)\X
    C_T(M)\Big).
  \end{equation*} For any $Y\subseteq M$, by $\mathcal{O}\cap Y$ we mean
  the set of all those points of $Y$ that, regarded as constant
  functions, lie in $\mathcal{O}$. Using this convention one has that
  $\Omega\cap Y = \mathcal{O}\cap Y$ and, in particular, $\Omega \cap M
  =\mathcal{O}\cap M$. This identity, together with Theorem \ref{SVDAEs:teo1}, 
  implies that
  \begin{equation*}
    \begin{aligned}
      | \deg(\Psi, \mathcal{O} \cap M)| &= |\deg(\Psi, \Omega \cap M)|
      \\
      &= |\deg(\Psi, (\Omega \cap U)\cap M)| = |\deg\big(F,\Omega \cap
      U\big)|\neq 0.
    \end{aligned}
  \end{equation*} 
  Hence, the hypotheses of Theorem \ref{UT:Theorem-3.3-spaSepVar} are
  satisfied, and we thereby obtain the existence of a connected subset
  $\Lambda$ of
  \begin{equation*} \big\{ (\lambda,(x,y))\in
    \mathcal{O}:\text{$(x,y)$ is a nonconstant solution of}\,\,
    \eqref{SVDAEs:eq-su-M}\big\},
  \end{equation*} whose closure in $\mathcal{O}$ is not
  compact and meets 
  $\left\{(0,(\su p,\su
    q))\in\mathcal{O}:\Psi(p,q)=(0,0)\right\}$.
  Observe that this set coincides with $\big\{(0,\su p,\su
  q)\in\Omega:F(p,q)=(0,0)\big\}$ which is the set of the trivial 
  $T$-periodic pairs to \eqref{SVDAEs:0a}.
  Moreover, from the equivalence of \eqref{SVDAEs:eq-su-M} with
  \eqref{SVDAEs:0a}, we have that each $(\lambda,(x,y))\in\Lambda$ is a
  nontrivial $T$-periodic pair to \eqref{SVDAEs:0a}.  Since
  $M$ is closed in $U$, it follows that any 
  relatively closed subset of $\mathcal{O}$ is relatively
  closed in $\Omega$ too and vice versa. 
  Thus, the closure of $\Lambda$ in $\mathcal{O}$ 
  coincides with the closure of $\Lambda$ in
  $\Omega$, and hence $\Lambda$ fulfills the 
  assertion.
\end{proof}

Under the extra assumption that $M=g^{-1}(0)$ is closed in $\rkrs$ ,
we are able to retrieve some further information about the connected
components of the set of $T$-periodic pairs to \eqref{SVDAEs:0a}.

\begin{lemma} \label{obs1} Let $U \subseteq \R^k \X \R^s$ be open and
  connected with $M\subseteq U$.  Assume that $M$ is closed in
  $\rkrs$.  Let $\Omega\subseteq [0, \infty)\X C_T(U)$ be open, and
  let $\Gamma\subseteq [0, \infty)\X C_T(M)$ be a connected component
  of the set of $T$-periodic pairs to \eqref{SVDAEs:0a} that meets
  $\{(0,(\su p,\su q))\in\Omega:F(p,q)=(0,0)\}$. Assume also that the
  intersection $\Gamma\cap \Omega$ is not compact.  Then, $\Gamma$ is
  either bounded or contained in $\Omega$.  Moreover, if $\Omega$ is
  bounded, then $\Gamma\cap \D\Omega\ne \emptyset$.
\end{lemma}
\begin{proof}
  Since $M$ is a closed subset of $\rkrs$, it follows that the metric
  space $[0, \infty)\X C_T(M)$ is complete.  Ascoli's Theorem implies
  that any bounded set of $T$-periodic pairs to \eqref{SVDAEs:0a} is
  totally bounded, which means relatively compact due to the
  completeness of $C_T(M)$.  As a straightforward consequence, since
  $\Gamma$ is closed, if $\Gamma$ is bounded then it is also compact.
  Therefore, $\Gamma$ cannot be both bounded and contained in
  $\Omega$.  The last part of the assertion follows from the fact that
  $\Gamma$ is connected and that $\emptyset\ne \Gamma \cap \Omega \ne
  \Gamma$.
\end{proof}

In particular, we have

\begin{corollary} Let $a$, $f$, $h$, $g$, $U$ and $F$ be as in Theorem
  \ref{SVDAEs:main-thm}.  Assume that $M$ is closed in $\rkrs$. Let
  $\Omega\subseteq [0, \infty)\X C_T (M)$ be open and such that
  $\deg(F,\Omega\cap M)$ is defined and nonzero.  Then there exists a
  connected component $\Gamma$ of $T$-periodic pairs that meets
  $\{(0,(\su p,\su q))\in\Omega:F(p,q)=0\}$, and cannot be both
  bounded and contained in $\Omega$. In particular, if $\Omega$ is
  bounded, then $\Gamma \cap \Omega \ne \emptyset$.
\end{corollary}
\begin{proof}
  Applying Theorem \ref{SVDAEs:main-thm} and Lemma \ref{obs1} the
  thesis follows readily.
\end{proof}

As a consequence of Theorem \ref{SVDAEs:main-thm} and Lemma
\ref{obs1}, we now establish the following continuation result

\begin{corollary}\label{SVDAEs:cor1}
  Let $a$, $f$, $h$, $g$, $U$ and $F$ be as in Theorem
  \ref{SVDAEs:main-thm}.  Assume that $M = g^{-1}(0)$ is closed in
  $\rkrs$. Let $V\subseteq U$ be open and such that $\deg(F,V)$ is
  well defined and nonzero. Then, there exists a connected component
  $\Gamma$ of $T$-periodic pairs to \eqref{SVDAEs:0a} that meets the
  set
  \begin{equation*}
    \big\{(0,(\su{p},\su{q}))\in [0,\infty)\times C_T(U): (p,q)\in V\cap
    F^{-1}(0,0)\big\}
  \end{equation*}
  and is either unbounded or meets
  \begin{equation*}
    \big\{(0,(\su{p},\su{q}))\in [0,\infty)\times C_T(U): (p,q)\in
    F^{-1}(0,0)\setminus V\big\}.
  \end{equation*}
\end{corollary}

\begin{proof}
  Consider the open subset $\Omega$ of $[0,\infty)\times C_T(U)$ given
  by
  \begin{equation*}
    \Big([0,\infty)\times C_T(U)\Big)
    \Big\backslash \Big\{ (0,(\cl{p},\su{q}))\in [0,\infty)\times
    C_T(U): (p,q)\in F^{-1}(0,0)\setminus V\Big\}.
  \end{equation*}
  Clearly, we have that $U\cap \Omega=V$ and hence $\deg(F,U\cap
  \Omega)\neq 0$.  Thus, Theorem \ref{SVDAEs:main-thm} implies the
  existence of a connected component $\Gamma$ of $T$-periodic pairs of
  \eqref{SVDAEs:0a} that meets the set of the trivial $T$-periodic
  pairs $\{(0,\su p,\su q)\in\Omega:F(p,q)=0\}$, and whose
  intersection with $\Omega$ is not compact. By Lemma \ref{obs1}, if
  $\Gamma$ is bounded, then it necessarily intersects the boundary of
  $\Omega$ which is given by
  \begin{equation*}
    \big\{ (0,(\su{p},\su{q}))\in [0,\infty)\times C_T(U): (p,q)\in
    F^{-1}(0,0)\setminus V\big\}.
  \end{equation*}
  Then, the conclusion follows.
\end{proof}

Now, we give some examples in order to illustrate our results.

\begin{example} \label{SVDAEs:firstExample} We examine the following
  scalar DAE:
  \begin{equation} \label{SVDAEs:0022} \left\{ \begin{array}{l} \dot x
        = \frac{a(t)}{b(t)}x + \lambda h(t, x, y), \\ g(b(t)x, y) = 0, \\
      \end{array} \right.
  \end{equation} where $a : \R\to\R$,  $b : \R\to (0, \infty)$, 
  $g : \R^2 \to \R$ and $h : \R \X \R^2  \to \R$ are
  continuous functions. We actually assume  
  that $b\in C^{\infty}$,
  $h$ is $C^1$, $g$ is $C^{\infty}$
  with the property that $\partial_2 g (p, q)$ is nonzero for all
  $(p,q)\in \R^2$.  Given $T>0$, we also require that $a, b$ are
  $T$-periodic and, that
  $h$ is $T$-periodic with respect to its first variable. 
  Consider the change of variables
  \begin{equation} \label{SVDAEs:ch-var} \Theta_b :\R\X \R^2\to \R\X
    \R^2,\,\, (t, (x, y))\mapsto (t, (\boldsymbol{\rmx}, y)),
  \end{equation} 
  with $\boldsymbol{\rmx}=b(t) x$.  Differentiating we get
  \begin{equation*} \dot x(t) =\frac{\dot{\boldsymbol{\rmx}}(t) -
      \dot{b}(t) x(t)}{b(t)}.
  \end{equation*} 
  Problem \eqref{SVDAEs:0022} can be equivalently rewritten in the
  following form
  \begin{equation} \label{SVDAEs:transformed-vinc-mob-1}
    \left\{ \begin{array}{l} \dot{\boldsymbol{\rmx}} =
        \frac{\dot{b}(t) + a(t)}{b(t)}\boldsymbol{ \rmx} + \lambda
        b(t) h(t,
        \frac{\boldsymbol{\rmx}}{b(t)}, y), \\ g(\boldsymbol{\rmx}, y) = 0. \\
      \end{array} \right.
  \end{equation} 
  Assume that $ \frac{1}{T}\int_0^T \frac{a(s)}{b(s)}ds\ne 0$.  Then,
  one has that
  \begin{equation*}
    \frac{1}{T}\int_0^T \frac{\dot{b}(s) + a(s)}{b(s)}ds = 
    \frac{1}{T}\int_0^T \frac{\dot{b}(s)}{b(s)}ds +
    \frac{1}{T}\int_0^T \frac{a(s)}{b(s)}ds =
    \frac{1}{T}\int_0^T \frac{a(s)}{b(s)}ds
    \ne 0,
  \end{equation*} 
  and the system of equations \eqref{SVDAEs:transformed-vinc-mob-1} is
  of type \eqref{SVDAEs:0a}.

  For instance, take $a(t)= |\cos(t)|$, $b(t)= 2+ \sin (t)$,
  $g(x,y)=y^5+y^3+y+x^3$, $T=2\pi$ and let $\Omega\subseteq
  [0,\infty)\X C_{2\pi}(\R^2)$ be the open set given by
  \begin{equation*}
    [0,\infty)\X \big\{(x,y)\in C_{2\pi}(\R^2) \colon 
    x(t)>-1/b(t),\, \forall \ t\in\R \big\}.
  \end{equation*}
  It can be easily checked that $\int_0^{2\pi}a(t)/b(t)dt=2 \ln (3)$.
  Referring to \eqref{SVDAEs:ch-var}, we take into account the induced
  transformation $\h{\Theta}_b\colon [0,\infty)\X C_{2\pi}(\R^2)\to
  [0,\infty)\X C_{2\pi}(\R^2)$ given by
  \begin{equation}\label{SVDAEs:theta-hat}
    \h{\Theta}_b(\lambda, (\varphi, \psi))(t) =\big(\lambda, (b(t)
    \varphi(t), \psi(t))\big), \,\, \forall t\in\R.
  \end{equation} 
  Since $\Theta_b$ is continuous and invertible with continuous
  inverse, then $\h{\Theta}_b$ enjoys the same properties and the
  inverse $\h{\Theta}_b^{-1}\colon [0,\infty)\X C_{2\pi}(\R^2)\to
  [0,\infty)\X C_{2\pi}(\R^2)$ is defined by
  \begin{equation} \label{SVDAEs:inv-utility} \h{\Theta}_b^{-1}(\mu,
    (\zeta, \omega))(t)= \Big(\mu, \Big(\frac{\zeta(t)}{b(t)},
    \omega(t)\Big) \Big), \,\, \forall t\in\R.
  \end{equation} 
  Through the transformation \eqref{SVDAEs:theta-hat}, the open set
  $\Omega$ becomes
  \begin{equation*}
    \h{\Omega}=[0,\infty)\X \big\{(\xbf, y)\in C_{2\pi}(\R^2) :\xbf
    (t)>-1,\, \forall \ t\in\R\big\}.
  \end{equation*}  
  With respect to the Equation \eqref{SVDAEs:transformed-vinc-mob-1},
  the map $F$ in \eqref{SVDAEs:F} is given by $F(\xbf,y)=(\xbf ,
  y^5+y^3+y+\xbf^3)$.  A direct computation shows that
  $F^{-1}((0,0))\cap\h{\Omega}$ consists of the singleton $\{(0, 0)\}$
  and that $\deg(F,\R^2\cap\h{\Omega})=1$.  Therefore, Theorem
  \ref{SVDAEs:main-thm} yields a connected set $\Pi \subseteq
  \h{\Omega}$ of nontrivial $2\pi$-periodic pairs to
  \eqref{SVDAEs:transformed-vinc-mob-1} emanating from the trivial
  $2\pi$-periodic pair $\{(0,(\su 0, \su 0))\}$, whose closure in
  $\h{\Omega}$ is noncompact.

  It is not difficult to prove that $\h{\Theta}_b^{-1}$ sends
  $2\pi$-periodic pairs of \eqref{SVDAEs:transformed-vinc-mob-1} into
  $2\pi$-periodic pairs of \eqref{SVDAEs:0022}. Moreover, the trivial
  $2\pi$-periodic pair $(0,(\su 0, \su 0))$ (which is the unique
  trivial $2\pi$-periodic pair to \eqref{SVDAEs:0022}) is sent to
  itself by $\h{\Theta}_b^{-1}$.  Thus, we can easily infer that
  $\Lambda :=\h{\Theta}_b^{-1}(\Pi)\subseteq \Omega$ is a connected
  set of $2\pi$-periodic pairs to \eqref{SVDAEs:0a}, whose closure in
  $\Omega$ is noncompact and meets $\{(0,(\su 0, \su 0))\}$.
\end{example}

\begin{example}
  Let us now analyze a separated variables DAE coming from a system of
  equations describing the process of heat generation in a exothermic
  chemical reactor. The considered equation, which is derived from a
  well known model (see, e.g. \cite{dae, PANT:1988}), is affected by a
  $T$-periodic perturbation due to the presence of a $T$-perturbation
  pair $(\alpha, h)$, $T>0$ given. In what follows $C_0$ represents
  the initial reactant concentration, $\mathsf{T}_0$ is the initial
  temperature, $\mathsf{T}_c$ is the cooling temperature, $c=c(t)$ and
  $\mathsf{T}=\mathsf{T}(t)$ are concentration and temperature at time
  $t$, and by $R=R(t)$ we denote the reaction rate for unit volume.
  The equation is the following
  \begin{equation} \label{RDAEs:exo-reac}
    \begin{pmatrix}
      \dot c \\
      \dot{\mathsf{T}} \\
      0      \\
    \end{pmatrix} = (1+\alpha(t))\begin{pmatrix}
      k_1 (C_0 - c) - R, \\
      k_1 (\mathsf{T}_0 - \mathsf{T}) +
      k_2R - k_3(\mathsf{T} - \mathsf{T}_c) \\
      R -  k_3 e^{-\frac{k_4 c}{\mathsf{T}}}\\
    \end{pmatrix}
    + \lambda
    \begin{pmatrix}
      h_1 (t, c, \mathsf{T}, R) \\
      h_2 (t, c, \mathsf{T}, R) \\
      0 \\
    \end{pmatrix},
  \end{equation}
  where $k_1, k_2, k_3$ and $k_4$ are given constants.  Here, $\alpha
  : \R \to \R$ is continuous with $\frac{1}{T}\int_0^T\alpha(t)dt =0$.
  The map $h = (h_1, h_2, h_3) \colon \R \times \R^3 \to \R\sp{3}$
  depends on time and on the state of the system, $h$ is $C^1$, and
  both $\alpha$ and $h$ are supposed to be $T$-periodic.  Notice that,
  when $\lambda=0$ and $\alpha\equiv 0$, the system of equations
  \eqref{RDAEs:exo-reac} describes the mentioned exothermic reactor
  model.
  
  Since we want to show how Theorem \ref{SVDAEs:main-thm} applies to
  \eqref{RDAEs:exo-reac}, we are interested in the equation itself
  regardless of its physical meaning. Thus, in what follows, we assume
  that all the quantities involved in \eqref{RDAEs:exo-reac} are
  dimensionless.

  Set $\mathbf{x} = (x_1, x_2) := (c, T)$ and $y := R$.  Let
  $U\subseteq \R^3$ be the open set given by $\{(\mathbf{x},y)\in
  \R^2\X\R \colon x_2>0\}$. Consider the maps $g \colon U \to \R$ and
  $f \colon U \to \R^2$ given by
  \begin{subequations}
    \begin{gather*}
      f(\mathbf{x}, y) = \left( \begin{array}{c|c} \!  A^{12} & A^3
          \! \end{array}\right)
      \begin{pmatrix}
        \mathbf{x} \\
        y \\
      \end{pmatrix} + B \,\quad \textrm{and} \,\quad g(\mathbf{x}, y)
      = y - k_3
      e^{-\frac{k_4 x_1}{x_2}}, \\
      \hspace{-11.8 cm}  \textrm{where}\\
      A^{12}=  \left( \begin{array}{cc} -k_1  & 0  \\
          0 & -(k_1 + k_3)
        \end{array}\right),\,\, 
      A^3 = \left( \begin{array}{cc} -1 \\ k_2 \end{array}\right) \,\,
      \textrm{and}\,\, B = \left( \begin{array}{c}
          k_1 C_0 \\
          k_1 \mathsf{T}_0 + k_3 \mathsf{T}_c \\
        \end{array} \right). 
    \end{gather*}
  \end{subequations}
  It is immediate to check that $\D_2 g \equiv 1$.  Let $F(\mathbf{x},
  y) = (f(\mathbf{x}, y), g(\mathbf{x}, y))$ be as in
  \eqref{SVDAEs:F}. Using the so called ``generalized Gauss
  algorithm'', we get
  \begin{equation*}
    \begin{aligned}
      \det \emph{d}_{(\mathbf{x}, y)}F = \det
      \left( \begin{array}{c|c}
          A^{12} & A^3 \\
          \hline
          \D_{1} g (\mathbf{x}, y) & 1\\
        \end{array} \right) 
      &    =  \det (A^{12} - A^3 \D_1 g (\mathbf{x}, y) ) \\
      & = (k_1 - \eta(\mathbf{x}))(k_1 + k_3) -
      k_1k_2 \eta(\mathbf{x})x_1/x_2,
    \end{aligned}
  \end{equation*} 
  where $\eta(\mathbf{x})= \frac{k_3k_4}{x_2}e^{\frac{-k_4x_1}{x_2}}$.
  Now, for instance, take $k_1= k_3=1/2, k_2=2, k_4=1$ and $C_0>
  (\mathsf{T}_0 + \mathsf{T}_c)/2 >0$. Under these hypotheses, a
  direct computation shows that $F^{-1}(0)$ consists exactly of a
  single point $(\mathbf{x}_0, y_0)$, and that $ \det
  \emph{d}_{(\mathbf{x}_0, y_0)}F \ne0$.  Hence, Theorem
  \ref{SVDAEs:main-thm} applies to the considered DAE, and we can
  conclude that there exists a connected set $\Lambda$ of nontrivial
  $T$-periodic pairs to \eqref{RDAEs:exo-reac} whose closure in
  $\Omega = [0, \infty)\X C_T(U)$ in noncompact, and that meets
  $\{\big(0, (\mathbf{x}_0, y_0)\big)\}$.
\end{example}

As last example, we take into account a problem which is an adapted
version of a retarded equation examined in \cite{BIS-SPA:2011}.

\begin{example} Let $A,E\in\R^{n\X n}$. Consider the following
  implicit differential equation
  \begin{equation} \label{SVDAEs:exx2} E \dot{\boldsymbol{\rmx}} =
    a(t)A \boldsymbol{\rmx} + \lambda C(t)S(\boldsymbol{\rmx}), \,\,\,
    \lambda\geq 0,
  \end{equation} 
  where $a : \R\to\R$, $S\colon \R^n\to\R^n$ and $C\colon\R\to\R^{n\X
    n}$ are continuous, with $a$ and $C$ $T$-periodic, $T>0$.  Assume
  also that $C$ and $E$ satisfies the following relations:
  \begin{equation*}
    \ker\, C^T(t) = \ker\, E^T,\;\forall \,t\in\R,\,\, 
    \text{and $n> \dim\ker\, E^T>0$}.
  \end{equation*}
  In particular, we have that $s=\rank\, E =\rank\, C(t)$ is a
  positive constant for all $t\in\R$.  Under these assumptions,
  \cite[Lemma\ 5.5]{BIS-SPA:2011} applies to the DAE
  \eqref{SVDAEs:exx2}, and so there exist orthogonal matrices $P, Q\in
  \R^{n\X n}$ that realize a singular value decomposition (SVD) for
  $E$, such that \eqref{SVDAEs:exx2} can be equivalently rewritten as
  \begin{equation}\label{svdeq1}
    PEQ^T\dot{\boldsymbol{\mathsf{x}}} =a(t)PAQ^T
    \boldsymbol{\mathsf{x}}+\lambda
    PC(t)Q^TQS(Q^T\boldsymbol{\mathsf{x}}),
  \end{equation} with $\boldsymbol{\rmx}=Q^T \boldsymbol{\mathsf{x}}$,
  \begin{subequations} \label{SVDAEs:PQ}
    \begin{gather*}
      PEQ^T = \begin{pmatrix} \widetilde{E}_s & 0 \\
        0 & 0 \end{pmatrix},\,\, PAQ^T
      = \begin{pmatrix} \widetilde A_{11} & \widetilde A_{12} \\
        \widetilde A_{21} & \widetilde A_{22} \end{pmatrix}, \,\,
      PC(t) Q^T = \begin{pmatrix} \widetilde{C}_s(t) & 0 \\ 0 &
        0 \end{pmatrix},
    \end{gather*} and, setting
    $\mathbf{p}=Q^T\boldsymbol{\mathsf{p}}$, $\mathbf{p}\in \R^n$, we
    also have
    \begin{equation*}
      QS (Q^T\boldsymbol{\mathsf{p}})
      = \begin{pmatrix} \widetilde{S}_s(Q^T\boldsymbol{\mathsf{p}}) \\
        \widetilde{S}_{n-r}
        (Q^T\boldsymbol{\mathsf{p}})\end{pmatrix}, 
    \end{equation*}
  \end{subequations}
  where $\widetilde{E}_s\in \R^{s\X s}$ is a diagonal matrix with
  positive diagonal elements, $ \widetilde{C}_s \in C\big(\R, \R^{s\X
    s})$ is nonsingular for any $t\in\R$, $ \widetilde{S}_s \in
  C\big(\R^n, \R^{s})$ $ \widetilde{S}_{n-s} \in C\big(\R^n,
  \R^{n-s})$, $\widetilde A_{11}\in \R^{s\X s}$ and $\widetilde
  A_{22}\in \R^{n-s\X n-s}$. Decompose the Euclidean space $\R^n$ with
  respect to its orthogonal subspaces $\im E^T\cong \R^s$ and $\ker
  E\cong \R^{n-s}$, so that $\R^n\simeq\R^s\X\R^{n-s}$, and set
  $\boldsymbol{\mathsf{x}}=(x,y)\in \R^s\X\R^{n-s}$.  Thus
  \eqref{svdeq1} becomes
  \begin{equation}
    \label{SVDAEs:exx3} \left\{ \begin{array}{l} 
        \widetilde{E}_s {\dot{{x}}} =a(t)
        \big(\widetilde A_{11} x + \widetilde A_{12} y) + \lambda 
        \widetilde{C}_s(t) S_s
        \big({x}\big),\,\,\, \lambda\geq 0, \\ 
        \widetilde A_{21} x + \widetilde A_{22} y =0.
      \end{array} \right. \\
  \end{equation} 
  Assume that $\widetilde A_{22}$ is invertible. Notice that $\rank
  \widetilde{A}_{22}$ remains the same for any other choice of
  matrices $ P $ and $ Q $ realizing a SVD for $E$, so the required
  condition does not depend on the chosen SVD for $E$.  Define $F(x,
  y) = (\widetilde A_{11} x + \widetilde A_{12} y, \widetilde A_{21} x
  + \widetilde A_{22} y)$, with $(x,y)\in \R^s\X\R^{n-s}$.  Therefore,
  Theorem \ref{SVDAEs:main-thm} ensures the existence of a connected
  subset $\Lambda$ of nontrivial $T$-periodic pairs to
  \eqref{SVDAEs:exx3} whose closure in $[0,\infty)\X
  C_T(\R^s\X\R^{n-s})$ is noncompact and meets the set
  $\big\{(0,(\su{p}, \su{q}))\in [0,\infty)\X C_T(\R^s\X\R^{n-s}) :
  F(p, q) \!=0\!\big\}$. Recasting the argument in \cite[Corollary\
  5.7]{BIS-SPA:2011}, it follows that $\Lambda$ generates an unbounded
  connected set of nontrivial $T$-periodic pairs to
  \eqref{RDAEs:exo-reac} emanating from $\big\{(0,\su{\mathbf{p}})\in
  [0,\infty)\X C_T(\R^n) : A \mathbf{p} \!=0\!\big\}$.
\end{example}

\section{A multiplicity result} \label{SVDAEs:section4}

In this section we give a multiplicity result that can be inferred
from Theorem \ref{SVDAEs:main-thm} and Corollary
\ref{SVDAEs:cor1}. For the reminder of this section $a, f, g, h, U, T$
and $F$ will be as in Section \ref{SVDAEs:section3}.  We will also
suppose that $\frac{1}{T}\int^T_0 a(t) dt=1$, for the sake of
simplicity.  The approach followed here leans on the argument given in
Theorem \ref{SVDAEs:main-thm} and on a local analysis regarding the
set of $T$-periodic solutions to \eqref{SVDAEs:0a}.

Let $(p_0, q_0)$ be an isolated zero of $F$. Then, since $\D_2 g(p_0,
q_0)$ is invertible, we can locally ``decouple'' \eqref{SVDAEs:0a}.
Namely, by the Implicit Function Theorem, there exist neighborhoods $V
\subseteq \R^k$ of $p_0$ and $W \subseteq \R^s$ of $q_0$, and a
$C^1$-function $\gamma : V \to \R^s$ such that $M \cap (V \X W)$ is
the graph of $\gamma$, where $M=g^{−1}(0)$.  Thus, within $V \X W$,
Equation \eqref{SVDAEs:0a} can be rewritten with $y=\gamma(x)$ as
follows:
\begin{equation*}
  \dot x = a(t) f\big(x, \gamma(x)\big) + 
  \lambda h\big(t,x, \gamma(x)\big).
\end{equation*}
Linearizing the above equation at $(p_0, q_0)$, for $\lambda=0$, we
get
\begin{equation} \label{SVDAEs:linearized} \dot{\xi} =
  a(t)\big[\partial_1 f(p_0,q_0) + \partial_2 f(p_0,
  q_0)d_{(p_0,q_0)}\gamma\big]\xi,\,\, \textrm{on}\,\, \R^k,
\end{equation}
which is a nonautonomous linear ODE in $\R^k$.

We will say that $(p_0, q_0)$ is a $T$-\emph{resonant zero} of $F$, if
\eqref{SVDAEs:linearized} admits nonzero $T$-periodic solutions.
Notice that this definition is analogous to the one given in
\cite{spaDAE} in the context of $T$-periodic perturbations to
autonomous semi-explicit DAEs of the form
\begin{equation} \label{SVDAEs:0-non-a} \left\{ \begin{array}{l} \dot
      x = f(x, y) +\lambda h(t, x, y),\,\,\, \lambda\geq 0, \\
      g(x, y) = 0, \\
    \end{array} \right.
\end{equation}
where $f,h$ and $g$ are as above.  \medskip

Let $ \Phi(p_0,q_0)$ be the linear endomorphism of $\R^k$ given by
\begin{equation*}
  \partial_1 f(p_0,q_0) 
  - \partial_2 f(p_0, q_0)[\partial_2 g(p_0,q_0)]^{-1}
  \partial_1 g(p_0,q_0), 
\end{equation*}
then Equation \eqref{SVDAEs:linearized} becomes
\begin{equation} \label{SVDAEs:linearized-rew} \dot{\xi} =
  a(t)\Phi(p_0,q_0)\xi, \,\, \textrm{on}\,\, \R^k.
\end{equation}
Applying to \eqref{SVDAEs:0-non-a} the same linearization procedure
used to obtain \eqref{SVDAEs:linearized-rew}, we get the following
autonomous ODE on $\R^k$
\begin{equation} \label{SVDAEs:linearized-non-a} \dot{\xi} = \Phi(p_0,
  q_0)\xi.
\end{equation}

\noindent The next result shows that the $T$-resonancy condition at
$(p_0, q_0)\in F^{-1}(0)$ does not depend on the presence of the
perturbation factor $a \colon \R \to \R$ in \eqref{SVDAEs:0a}. Namely,
it holds that

\begin{proposition} \label{SVDAEs:equivalence-T-res} Let $(p_0,q_0)\in
  M$ be a zero of $F=(f,g)$. Then $(p_0,q_0)$ is $T$-resonant for
  \eqref{SVDAEs:0a} if and only if it is $T$-resonant for
  \eqref{SVDAEs:0-non-a}.
\end{proposition}
\begin{proof}
  Let $P^{a\Phi}_\tau$ and $P^{\Phi}_\tau$ be the local Poincar\'e
  $\tau$-translation operators, $\tau\in \R$, associated to the
  equations \eqref{SVDAEs:linearized-rew} and
  \eqref{SVDAEs:linearized-non-a} respectively.
  
  Assume that $(p_0, q_0)$ is a $T$-resonant zero for the equation
  \eqref{SVDAEs:0a}, and let $\xi_0\in \R^k$ be an initial point of a
  $T$-periodic solution $\xi$ to \eqref{SVDAEs:linearized-rew}.  In
  such a case, $P^{a\Phi}_T(\xi_0)$ is defined and it holds true that
  $P^{a\Phi}_T(\xi_0)=\xi_0$.  Arguing as in Remark
  \ref{SVDAEs:remark-on-ivp}, it follows that also $P^{\Phi}_T
  (\xi_0)$ is defined and that $P^{\Phi}_T(\xi_0)=
  P^{a\Phi}_T(\xi_0)=\xi_0$. This means that $\xi_0$ is an initial
  point of a nonzero $T$-periodic solution of $\dot \xi = \Phi(p_0,
  q_0)\xi$. Hence, $(p_0, q_0)$ is $T$-resonant for the Equation
  \eqref{SVDAEs:0-non-a} too.  The converse implication is
  straightforward.
\end{proof}

\begin{remark}
  The $T$-resonancy condition at $(p_0, q_0)$ can be read on the
  spectrum $\sigma (\Phi (p_0, q_0))$ of $\Phi(p_0, q_0)$. Indeed, the
  (unique) solution to the the Cauchy problem
  \begin{equation*} \label{SVDAEs:cauchy-lin} \dot{\xi} =
    a(t)\Phi(p_0,q_0)\xi, \,\,\, \xi(0)=\xi_0
  \end{equation*}
  is given by
  \begin{equation*}
    \xi = e^{-\int_0^ta(s)ds\Phi(p_0, q_0)}\xi_0.
  \end{equation*}
  Hence, $\xi$ is a $T$-periodic solution to \eqref{SVDAEs:linearized}
  if and only if $\xi_0\in \ker (I-e^{T\Phi(p_0,q_0)})$, where $I$ is
  the identity on $\R^k$.  Thus, $(p_0,q_0)$ is $T$-resonant if and
  only if, for some $n \in \Z$ one has that $\frac{2n\pi i}{T}\in
  \sigma(\Phi (p_0, q_0))$, with $i$ the imaginary unit.  Again, using
  the generalized Gauss algorithm, we obtain
  \begin{equation*}
    \begin{aligned}
      \det \emph{d}_{(p_0,q_0)}F = \det \big(\D_2g(p_0, q_0)\big)\cdot
      \det\, (\Phi(p_0,q_0)),
    \end{aligned}
  \end{equation*}
  so, if $(p_0, q_0)$ is non-$T$-resonant, then it is a nondegenerate
  zero of $F$.  In particular, we have that $\emph{i} (F, (p_0,
  q_0))\ne 0$.
\end{remark}

We have the following lemma

\begin{lemma} \label{SVDAEs:mr} Suppose that $a, f, g, h, U, T$ and
  $F$ be as in Theorem \ref{SVDAEs:main-thm}.  Let $(p_0,q_0)$ be a
  non-$T$-resonant zero of $F$. Then
  \begin{itemize}
  \item[(1)] the trivial $T$-periodic pair $(0, (\overline{p}_0 ,
    \overline{q}_0))$ is isolated in the set of $T$-periodic pairs
    corresponding to $\lambda = 0$;
  \item[(2)] there exists a connected set of nontrivial $T$-periodic
    pairs to \eqref{SVDAEs:0a} whose closure in $[0, \infty) \X C_T (U
    )$ contains $(0, (\overline{p}_0 , \overline{q}_0))$ and is either
    noncompact or intersects
    \begin{equation*}
      \{ (0, (\cl{p},\su{q}))\in [0,\infty)\times
      C_T(U): (p,q)\in F^{-1}(0,0)\}\backslash 
      \{(0, (\overline{p}_0 , \overline{q}_0 ))\}.
    \end{equation*}
  \end{itemize}
\end{lemma}
\begin{proof}
  We apply Proposition \ref{SVDAEs:equivalence-T-res} and \cite[Lemma\
  5.5]{spaDAE}.
\end{proof}
\medskip

Here, we fix some further notation.  Let $Y$ be a metric space and $X$
a subset of $[0, \infty )\X Y$.  Given $\mu\geq 0$, we denote by
$X_\mu$ the slice $\{y \in Y\colon (\mu, y) \in X\}$.  Using the
convention introduced in \eqref{SVDAEs:graph}, $Y$ will be identified
with the subset $\{0\} \X Y \subseteq [0, \infty)\X Y$.

Let X be a subset of $[0, \infty)\X Y$.  We say that a subset $A$ of
$X_0$ is an \emph{ejecting set} (for $X$) if it is relatively open in
$X_0$ and there exists a connected subset of $X$ which meets $A$ and
is not included in $X_0$. In particular, we say that $(p_0, q_0)\in
X_0$ is an ejecting point for $X$ if $\{(p_0, q_0)\}$ is an ejecting
set.  Notice that $(p_0, q_0)$ is isolated in $X_0$ being $\{(p_0,
q_0)\}$ open in $X_0$.  \smallskip

We now give a sufficient condition for the existence of ejecting
points for the set of the $T$-periodic pairs of \eqref{SVDAEs:0a}.

\begin{corollary}
  Let $(p_0, q_0)$ be a non-$T$-resonant zero of $F \colon U \to
  \rkrs$. Then, $(p_0, q_0)$ (regarded as the trivial $T$-periodic
  pair $(0, (\overline{p}_0, \overline{q}_0))$) is an ejecting point
  for the set of the $T$-periodic pairs of \eqref{SVDAEs:0a}.
\end{corollary}
\begin{proof} Let $X$ be the subset of $[0,\infty)\X C_T(U)$ of all
  the $T$-periodic pairs to \eqref{SVDAEs:0a}.  Due to Lemma
  \ref{SVDAEs:mr}, the non-$T$-resonant zero $(p_0, q_0)$ results to
  be an isolated point of $X_0$, which means that it is ejecting.
\end{proof}

We are now ready to establish the following result

\begin{proposition} \label{SVDAEs:last-prop} Let $a, f, g, h, U, T$
  and $F$ be as in Theorem \ref{SVDAEs:main-thm}. Assume also that $M
  = g^{-1}(0)$ is closed in $\rkrs$.  Let $\{(p_i, q_i)\}, i=1,\ldots,
  r,$ be non-$T$-resonant zeros of $F$ such that
  \begin{equation*}
    \deg(F,U) \ne \sum^r_{j=1} \emph{i} (F, (p_j , q_j)).
  \end{equation*}
  Suppose that \eqref{SVDAEs:0} does not admit an unbounded connected
  set of $T$-periodic solutions in $C_T (U)$. Then, there are at least
  $r+1$ different $T$-periodic solutions of \eqref{SVDAEs:0a} when
  $\lambda > 0$ is sufficiently small.
\end{proposition}
\begin{proof}
  We apply Proposition \ref{SVDAEs:equivalence-T-res} and
  \cite[Proposition\ 5.7]{spaDAE}.
\end{proof}

\begin{example} 
  Consider the following DAE in $U = (-\frac{3}{2}, \infty)\X
  (-\frac{3}{2}, \infty)$
  \begin{equation} \label{SVDAEs:mult-ex} \left\{ \begin{array}{l}
        \dot x = a(t) \big(xy^2 - x^2 y) + \lambda h(t, x, y),\\
        y - x^3 + \varepsilon y^3 = 0,
      \end{array}\right.
  \end{equation}
  where $0<\varepsilon<1$ is a given constant, $a\colon \R \to (0,
  \infty)$ and $h \colon \R\X (\R \X\R) \to \R$ are continuous and
  $T$-periodic, with respect to the variable $t$. Actually, we require
  that $h\in C^1$.  As usual, we assume that
  $\frac{1}{T}\int_0^Ta(t)dt=1$.  It is immediate to check that the
  zero set of $F(p, q) = (pq^2- p^2q, q- p^3 + \varepsilon q^3)$ is
  formed by the points $(0, 0)$, $(a_\varepsilon, a_\varepsilon )$ and
  $(- a_\varepsilon, -a_\varepsilon)$, where $a_\varepsilon :=
  1/\sqrt{1-\varepsilon}$, and that just the first of them is
  $T$-resonant.  Again, it can be easily seen that, in the case
  $\lambda = 0$, the only possible $T$-periodic solutions of
  \eqref{SVDAEs:mult-ex} are those corresponding to the zeros of $F$.

  Consider the homotopy $H \colon U \X [0, 1] \to \R \X \R$, defined
  by
  \begin{equation*}
    H(p, q; \rho) = (q^2p - p^2q, q - p^3 +(1 -\rho)\varepsilon q^3 - \rho), 
  \end{equation*}
  which is admissible in the sense that the set $\{(p,q)\in U \colon
  H(p,q; \rho)=0 \,\, \textrm{for}\,\, \textrm{some}\,\, \rho\}$ is
  compact.  Obviously, $H(p, q; 1) \ne 0$ for any $(p, q) \in U$, then
  it follows that $\deg (H(\cdot, \cdot; 1),U) = 0$.  By the Homotopy
  Invariance property of the degree, one has that $\deg (F,U) = 0$.

  Since $( a_{\varepsilon}, a_{\varepsilon})$ and $(- a_{\varepsilon},
  - a_{\varepsilon})$ are non-$T$-resonant zeros of $F$, they are
  nondegenerate, and in particular it holds true that $\emph{i} (F, (
  a_{\varepsilon}, a_{\varepsilon})) = \emph{i} (F, (-
  a_{\varepsilon}, - a_{\varepsilon})) =-1$.  Therefore, Proposition
  \ref{SVDAEs:last-prop} ensures the existence of at least three
  $T$-periodic solutions of \eqref{SVDAEs:mult-ex}, for sufficiently
  small $\lambda > 0$.
\end{example}
\smallskip

\noindent\textbf{Acknowledgements.}
The author would like to thank M.\ Spadini for many valuable
discussions and suggestions.

\end{document}